\documentclass[12pt]{amsart}

\usepackage{amsthm}
\newtheorem{theorem}{Theorem}
\newtheorem{lemma}[theorem]{Lemma}
\newtheorem{proposition}[theorem]{Proposition}
\theoremstyle{definition}

\theoremstyle{remark}
\newtheorem{remark}[theorem]{Remark}

\usepackage{amssymb}
\usepackage{empheq}
\usepackage{stmaryrd}
\usepackage{enumerate}
\usepackage[shortlabels]{enumitem}
\usepackage{calc}
\usepackage{url}

\usepackage{xcolor}

\newcommand{\norm}[1]{\left\lVert#1\right\rVert}
\usepackage{mathtools}
\DeclarePairedDelimiter{\ceil}{\lceil}{\rceil}

\newcommand{\CAT}{{\rm{CAT}(0)}}
\newcommand{\ql}[2]{\left\langle\overrightarrow{#1},\overrightarrow{#2}\right\rangle}

\usepackage{mathabx}

\usepackage[left=2.0cm,%
                right=2.0cm,%
                top=2.5cm,%
                bottom=3.5cm,%
                headheight=12pt,%
                a4paper]{geometry}%

\begin{document}

\title[On the proximal point algorithm for strongly quasiconvex functions]{On the proximal point algorithm for strongly quasiconvex functions in Hadamard spaces}

\author{Nicholas Pischke}
\date{\today}
\maketitle
\vspace*{-5mm}
\begin{center}
{\scriptsize Department of Mathematics, Technische Universit\"at Darmstadt,\\
Schlossgartenstra\ss{}e 7, 64289 Darmstadt, Germany, \ \\ 
E-mail: pischke@mathematik.tu-darmstadt.de}
\end{center}

\maketitle
\begin{abstract}
We prove the convergence of the proximal point algorithm for finding the unique minimizer of a strongly quasiconvex function in general nonlinear Hadamard spaces, generalizing a recent result due to F.\ Lara. Our argument is rather elementary and brief and relies only on a few properties of strongly quasiconvex functions and their proximal operators which are established here for the first time over these nonlinear spaces. In particular, our convergence proof is fully effective and actually yields fast (ranging up to linear) rates of convergence for the iterates towards the solution and for the function values towards the minimum. These rates are novel even in the context of Euclidean spaces.
\end{abstract}

\section{Introduction}

One of the most fundamental problems in optimization is the minimization
\[
\min_{x\in C}f(x)
\]
of a function $f:X\to\mathbb{R}$ over a certain space $X$ and a designated set of constraints $C\subseteq X$. Naturally, suitable restrictions have to be placed on the space and the function to make this problem feasible and the most classical formulation is the minimization of a convex and lower semicontinuous function over a Hilbert space. In that case, the most prominent method for approaching the corresponding minimization problem is the proximal point method originally devised in the work of Martinet \cite{Mar1970}, Rockafellar \cite{Roc1976} as well as Br\'ezis and Lions \cite{BL1978}, which constructs a next iterate by solving the proximal subproblem
\[
\mathrm{argmin}_{y\in C}\left\{f(y)+\frac{1}{2c_k}\norm{y-x^k}^2\right\}
\]
in terms of the current iterate $x^k$ and a regularization parameter $c_k>0$, which becomes tractable since $f(\cdot)+\frac{1}{2\beta}\norm{\cdot-x}^2$ is strongly convex for any point $x$ and $\beta>0$ if the function $f$ is convex.\\

When the assumption of convexity of $f$ is relaxed, already these proximal subproblems (and hence the proximal point method in general) become considerably harder and a great amount of research has been carried out to isolate certain classes of functions and spaces where the proximal point method remains effective. The most broad concrete class of functions we want to consider here is that of quasiconvex functions. This considerably wide class of (potentially nonconvex) functions however suffers from a range of problems in regards to the proximal point algorithm. As highlighted in \cite{Lar2022}, already on the real line quasiconvexity does not imply that the proximal subproblems are uniquely solvable (and in particular $f(\cdot)+\frac{1}{2\beta}\norm{\cdot-x}^2$ is not strongly quasiconvex for quasiconvex functions $f$). Also as highlighted in \cite{Lar2022}, the proximal point method is not really well-motivated as fixed points of the proximal map are not necessarily minimizers (and so the classical stopping criteria do not apply). As a result, while the proximal point method for quasiconvex functions has been heavily investigated (see e.g.\ \cite{BNCO2012,PC2007,Qui2024,QCM2020,QO2009,QO2012a,QO2012b,QMO2015}, among many more), essentially all convergence results rely on taking $c_k\to\infty$ which, as highlighted in \cite{Lar2022}, essentially reduces the proximal point method to a penalty method.\\

Starting at least with the recent work of Lara \cite{Lar2022}, an emphasis has thus been placed on strongly quasiconvex functions (which go back at least to the work of Polyak \cite{Pol1966}), a subclass of quasiconvex functions where (some of) these issues are avoided and one nevertheless encompasses interesting nonconvex (or at least not strongly convex) functions such as the Euclidean norm, which is strongly quasiconvex (as shown by Jovanovi\v{c} \cite{Jov1996}) without being strongly convex on any bounded convex set, and even the square root of the Euclidean norm which is strongly quasiconvex on each ball (as shown by Lara \cite{Lar2022}) without being differentiable (and which implies that e.g.\ $\max\{\sqrt{\norm{\cdot}},\norm{\cdot}^2-k\}$ for $k\in\mathbb{N}$ is strongly quasiconvex without even being convex as highlighted in \cite{Lar2022}). We refer to \cite{Lar2022} for further, very comprehensive, discussions on how this class is differentiated (or related) to various other classes of (non)convex functions. So, while some complications remain in this class, such as the proximal map not being single-valued, Lara has shown in \cite{Lar2022} that the natural analogue of the proximal point method, obtained by choosing some minimizer of said map at each stage, allows for a convergence result where the iteration converges to the unique minimizer of the function and the function values converge to the minimum, under the weak assumption that the regularization parameters $(c_k)$ are positively bounded below. This method and the corresponding convergence result has since then been extended in various directions by Lara and his co-authors (see e.g.\ \cite{GLM2023,GLM2024,IL2022}).\\

In this paper, we extend the convergence result for the proximal point method as established by Lara in \cite{Lar2022} to Hadamard spaces, that is complete geodesic metric spaces of nonpositive curvature, one of the most important classes of nonlinear hyperbolic spaces (see the next section for a more detailed introduction). On various subclasses of these spaces, the proximal point method for convex functions was previously investigated by a range of authors, culminating in the seminal work of Ba\v{c}\'ak \cite{Bac2013} who proved the weak convergence of the proximal point method in general Hadamard spaces (under a suitable notion of weak convergence, as will also be discussed later). We refer in particular to the comprehensive discussion therein for an excellent overview of the related literature. Previous work on proximal point methods for quasiconvex functions in nonlinear contexts is also rather abundant but, to our knowledge, largely limited to the setting of certain classes of Riemannian manifolds as mainly championed in the work of Papa Quiroz and his co-authors (see e.g.\ \cite{Qui2024,QCM2020,QO2012a,QO2012b} among many others) and further, in these works no special subclasses of quasiconvex functions are considered. In particular, as highlighted in the recent work \cite{Qui2024}, extensions of Lara's results from \cite{Lar2022} have so far not been developed in any kind of nonlinear context, a gap we try to fill with this paper.\\

Beyond however merely extending said results, we even provide a quantitative convergence result and, in particular, give very fast (ranging up to linear) rates of convergence for the sequence towards the solution and for the function values towards the minimum. These rates are moreover highly uniform and depend only on very few data relating to the iteration and the function. To our knowledge, these rates are novel even in the context of Lara's original setting of Euclidean spaces. We further want to highlight that our convergence proof is rather brief and elementary, relying only on a very small selection of key properties of strongly quasiconvex functions and their proximal maps that are generalized from the Euclidean context to the general setting of Hadamard spaces in this work as well. This elementarity of the convergence proof is key here for overcoming the obstacles of extending the results to general Hadamard spaces as Lara's proof given in \cite{Lar2022} relies on sequential compactness arguments in Euclidean spaces which are not available in general Hadamard spaces (and this similarly differentiates us from the setting of Riemannian manifolds which also are locally compact). In particular, it should be noted that our convergence result is ``strong'', i.e.\ is phrased not with a suitable notion of weak convergence but relative to the metric. This is not true for the usual proximal point method already for convex functions in Hilbert spaces as shown by G\"uler \cite{Gue1991} and is in this setting essentially due to the fact that strongly quasiconvex functions have a unique minimizer. In particular, our results thereby also extend Lara's results to infinite dimensional Hilbert spaces and show that the convergence in these spaces is strong as well. While not investigated here, it remains an interesting problem if similar, more elementary, proofs can be developed for the range of extensions of Lara's convergence result as surveyed above and if these then can be likewise transferred to nonlinear contexts as broad as Hadamard spaces, if appropriate.\\

All of the present results have been obtained using the methodology of proof mining, a program in mathematical logic that aims to classify and extract the computational content of prima facie ``non-computational'' proofs. We refer to the seminal monograph \cite{Koh2008} for a comprehensive overview of the subject and \cite{Koh2019} for a survey of more recent developments. However, while this logical perspective of proof mining was essential for deriving the results of this paper, we want to stress that the whole paper is presented without any explicit reference to logic and does not require any such background.\\

The rest of the paper is now organized as follows: Section \ref{sec:prelim} introduces the main necessary (but very minimal) background on Hadamard spaces. Section \ref{sec:prox} discusses strongly quasiconvex functions and their proximal maps in Hadamard spaces and establishes their key properties required for the main (quantitative) convergence results which are presented in Section \ref{sec:ppa}.

\section{Preliminaries}\label{sec:prelim}

All our results take place over various classes of nonlinear hyperbolic spaces, the most basic of which we introduce now:\\

A triple $(X, d, W)$ is called a hyperbolic space \cite{Koh2005} if $(X,d)$ is a metric space and $W : X \times X \times [0, 1] \to X$ is a function satisfying for all $x, y, z, w \in X$ and $\lambda, \lambda'\in [0,1]$:
\begin{enumerate}
\item[(W1)] $d(W(x,y,\lambda),z)\leq (1-\lambda)d(x,z) + \lambda d(y,z)$,
\item[(W2)] $d(W(x,y,\lambda), W(x,y,\lambda'))=|\lambda-\lambda'|d(x,y)$,
\item[(W3)] $W(x,y,\lambda)=W(y,x,1-\lambda)$,
\item[(W4)] $d(W(x,y,\lambda),W(z,w,\lambda))\leq (1-\lambda)d(x,z)+\lambda d(y,w)$.
\end{enumerate}
To ease notation in the following, we write $(1-\lambda) x \oplus \lambda y$ for the point $W(x, y, \lambda)$. We further call a set $C \subseteq X$ convex if $(1-\lambda) x \oplus \lambda y \in C$ for all $x,y\in C$ and $\lambda \in[0,1]$.\\

This class of hyperbolic spaces, which was originally introduced in \cite{Koh2005} as a logically motivated amalgamation of various notions of convex metric and hyperbolic spaces from the literature such as Takahashi's convex metric spaces \cite{Tak1970}, the hyperbolic spaces of Reich and Shafrir \cite{RS1990} as well as the spaces of hyperbolic type as introduced by Goebel and Kirk \cite{GK1983}, just to name a few, is commonly invoked as the appropriate nonlinear generalization of convexity from normed contexts and in particular allows for a nice direct approach to $\CAT$ and Hadamard spaces, the main classes of spaces that we are concerned with in this paper.\\

Concretely, in this paper, a $\CAT$ space is a hyperbolic space $(X,d,W)$ which satisfies
\[
d^2\left(\frac{1}{2} x\oplus \frac{1}{2}y,z\right)\leq \frac{1}{2} d^2(x,z)+\frac{1}{2}d^2(y,z)-\frac{1}{4}d^2(x,y)\tag{$\mathrm{CN}^-$}
\]
for all $x, y, z \in X$.\footnote{$\CAT$ spaces are often introduced as geodesic metric spaces satisfying the Bruhat-Tits $\mathrm{CN}$-inequality \cite{BT1972}. Any hyperbolic space that satisfies the $\mathrm{CN}$-inequality is hence a $\mathrm{CAT}(0)$-space and conversely, as every $\mathrm{CAT}(0)$-space is actually uniquely geodesic, any $\mathrm{CAT}(0)$-space is also a hyperbolic space by defining $W(x,y,\lambda)$ via the unique geodesic joining $x$ and $y$. Hence being a $\mathrm{CAT}(0)$-space in the usual sense is equivalent to the definition given here as, in the presence of (W1) -- (W4), the inequality $\mathrm{CN}^-$ is equivalent to the $\mathrm{CN}$-inequality. For more details, we refer to the discussions in \cite{GeK2008,Koh2005}.}\\

These spaces were originally introduced by Aleksandrov \cite{Ale1951} and were named $\CAT$ spaces by Gromov \cite{Gro1987}. A complete $\CAT$ space is called a Hadamard space. Examples of such spaces include Hilbert spaces, $\mathbb{R}$-trees and complete simply connected Riemannian manifolds of nonpositive sectional curvature, among many more. We refer to the seminal monograph \cite{BH1999} for a comprehensive overview of $\CAT$ and Hadamard spaces and further refer to \cite{Bac2014} for a shorter treatment focused on aspects of convex analysis and optimization.\\

We shortly note that the characterizing relation $\mathrm{CN}^-$ extends to arbitrary convex combinations:

\begin{lemma}[folklore, see e.g.\ \cite{DP2008}]
Let $(X,d,W)$ be a $\CAT$ space. Then
\[
d^2((1-\lambda) x\oplus \lambda y,z)\leq (1-\lambda) d^2(x,z)+\lambda d^2(y,z)-\lambda(1-\lambda)d^2(x,y)
\label{cn_ineq}\tag{$\mathrm{CN}^+$}
\]
for any $x,y,z\in X$ and $\lambda\in [0,1]$.
\end{lemma}

Crucially, we also rely on another equivalent characterization of $\CAT$ spaces as developed in the work of Berg and Nikolaev \cite{BN2008}. Concretely, define the so-called quasi-linearization function by
\[
\ql{xy}{uv}:=\frac{1}{2}\left( d^2(x,v) + d^2(y,u) - d^2(x,u)-d^2(y,v) \right)
\]
for all $x,y,u,v\in X$, using $\overrightarrow{xy}$ to denote pairs $(x,y)\in X^2$. As shown in \cite{BN2008}, this function is the unique function $X^2\times X^2\to\mathbb{R}$ in any metric space $(X,d)$ such that
\begin{enumerate}
\item[(1)] $\ql{xy}{xy} = d^2(x, y)$, 
\item[(2)] $\ql{xy}{uv}=\ql{uv}{xy}$,
\item[(3)] $\ql{xy}{uv}=-\ql{yx}{uv}$,
\item[(4)] $\ql{xy}{uv}+\ql{xy}{vw}=\ql{xy}{uw}$.
\end{enumerate}
for all $x,y,u,v,w\in X$. It then follows from the results in \cite{BN2008} that a hyperbolic space $(X,d,W)$ is a $\CAT$ space if, and only if,
\[
\ql{xy}{uv}\leq d(x,y)d(u,v)\tag{CS}\label{CS}
\]
for all $x,y,u,v\in X$, i.e.\ where a metric version of the Cauchy-Schwarz inequality holds. By these properties, this function can be seen as a nonlinear generalization of the inner product in Hilbert spaces. Beyond these properties, we will also rely on the following equality (which is immediate from the definition of the quasi-linearization function):

\begin{lemma}\label{lem:quasiLinLem}
For any metric space $(X,d)$ and $x,y,z\in X$:
\[
d^2(x,y)=d^2(x,z) + d^2(z,y) + 2\ql{xz}{zy}.
\]
\end{lemma}

\section{(Strongly) quasiconvex functions on Hadamard spaces}\label{sec:prox}

For the rest of this paper, unless said otherwise, we let $(X,d,W)$ be a Hadamard space, i.e.\ a complete $\CAT$ space. We begin with our main class of functions: A function $f:X\to \mathbb{R}$ is called
\begin{enumerate}
\item \emph{quasiconvex} if
\[
f((1-\lambda)x\oplus \lambda y)\leq\max\{f(x),f(y)\}
\]
for all $x,y\in X$ and $\lambda\in [0,1]$,
\item \emph{strongly quasiconvex} with value $\gamma>0$ if
\[
f((1-\lambda)x\oplus \lambda y)\leq\max\{f(x),f(y)\}-\lambda(1-\lambda)\frac{\gamma}{2}d^2(x,y)
\]
for all $x,y\in K$ and $\lambda\in [0,1]$.
\end{enumerate}

For the rest of this paper, we are interested in solving the problem $\min_{x\in X}f(x)$ for a strongly quasiconvex function $f$. We write $\mathrm{argmin}f$ for the set of minimizers of $f$. Crucially, we first note that for such functions the minimizer is unique:

\begin{lemma}\label{lem:minUnique}
Let $f:X\to\mathbb{R}$ be a strongly quasiconvex function with value $\gamma>0$. Then $\mathrm{argmin}f$ is at most a singleton.
\end{lemma}
\begin{proof}
If $x,y$ are two minimizers, define $z=1/2x\oplus 1/2y$. As $f$ is strongly quasiconvex, we get
\[
f(z)\leq\max\{f(x),f(y)\}-\frac{\gamma}{8}d^2(x,y)=\min f-\frac{\gamma}{8}d^2(x,y)
\]
and so $d^2(x,y)\leq (\min f - f(z))\cdot 8/\gamma \leq 0$, i.e.\ $x=y$.
\end{proof}

\begin{remark}
While all properties and results on quasiconvex functions given in \cite{Lar2022} are formulated relative to a closed and convex subset of the domain $\mathrm{dom}f$ of a function $f:\mathbb{R}^n\to\overline{\mathbb{R}}$, we omit these restrictions in this paper as any closed and convex set $C\subseteq X$ of a Hadamard space $X$ is naturally a Hadamard space itself and so we would not gain any further generality from considering closed and convex subsets of the domain of a function with extended real values.
\end{remark}

The main associated object with such (strongly) quasiconvex functions $f$ that we employ for approximating minimizers thereof is the proximal map (also sometimes called the Moreau-Yosida resolvent)
\[
\mathrm{Prox}_{\beta f}(x):=\mathrm{argmin}_{y\in X}\left\{f(y)+\frac{1}{2\beta}d^2(x,y)\right\}
\]
for $\beta>0$ and $x\in X$, used in a metric context without linear structure for the first time in \cite{Jos1995}.\\

If the function $f$ is lower semincontinuous (lsc), i.e.\
\[
\liminf_{n\to\infty} f(x_n)\geq f(x)
\]
for any $x\in X$ and any sequence $(x_n)$ such that $x_n\to x$, and additionally convex, then $\mathrm{Prox}_{\beta f}(x)$ is always a singleton (see e.g.\ \cite{Jos1995}). Now, in the quasiconvex case, the proximal map is in general not single valued already in Euclidean contexts (see e.g.\ the discussion in \cite{Lar2022}) but at least always non-empty (which in \cite{Lar2022} is reduced to the lower semicontinuity and supercoercivity of the object function of the proximal map and hence relies on a compactness argument).

Before moving on to the key properties of the proximal map on which our convergence proof relies, we now first show that the same is true in Hadamard spaces for quasiconvex lsc functions $f$. The respective argument relies similarly on a suitable notion of supercoercivity in metric spaces together with a weak compactness argument and we first introduce the necessary ingredients for this.\\

A function $f:X\to\mathbb{R}$ is called supercoercive at $x\in X$ if
\[
\lim_{d(x,y)\to \infty}\frac{f(y)}{d(x,y)}=\infty.
\]
A similar notion on Hadamard manifolds appears in \cite{QO2009} under the name of 1-coercivity. Clearly, if $f$ is bounded below, then the function $f(y)+\frac{1}{2\beta}d^2(x,y)$ is supercoercive at $x$. Crucially, we have the following result:

\begin{lemma}[folklore]\label{lem:supercoBound}
If $f:X\to\mathbb{R}$ is bounded below and supercoercive at $x\in X$, then any sequence $(y_n)$ with $f(y_n)\to\inf f$ is bounded.
\end{lemma}
\begin{proof}
Suppose not, i.e.\ let $(y_n)$ be an unbounded sequence such that $f(y_n)\to\inf f$. Then $d(x,y_n)\to\infty$ and so $f(y_n)/d(x,y_n)\to \infty$ as $n\to\infty$ since $f$ is supercoercive at $x$ but $f(y_n)/d(x,y_n)\to 0$ as $n\to\infty$ since $f(y_n)\to\inf f$, a contradiction.
\end{proof}

Weak convergence in $\CAT$ spaces goes back to the work of Jost \cite{Jos1994} and is often called $\Delta$-convergence, as it can be equivalently described via an extension of Lim's notion of $\Delta$-convergence \cite{Lim1976} to $\CAT$ spaces developed by Kirk and Panyanak \cite{KP2008} (as shown by Esp\'inola and Fern\'andez-Le\'on \cite{EFL2009}, see also the discussion in \cite{Bac2013}).

Concretely, define the asymptotic radius of a sequence $(x_n)\subseteq X$ at a point $x\in X$ via
\[
r(x_n,x):=\limsup_{n\to\infty}d(x_n,x)
\]
and define the general asymptotic radius of $(x_n)$ by
\[
r(x_n):=\inf_{x\in X}r(x_n,x).
\]
A point $x\in X$ is called an asymptotic center of $(x_n)$ if 
\[
r(x_n,x)=r(x_n)
\]
and in Hadamard spaces, asymptotic centers exist and are unique (see e.g.\ \cite{DKS2006}). A sequence $(x_n)\subseteq X$ is said to weakly converge to $x\in X$ if $x$ is the asymptotic center of each subsequence of $(x_n)$ and we write $x_n\to^w x$ in that case. A point $x\in X$ is a weak cluster point of a sequence $(x_n)$ if there is a subsequence $(x_{n_k})$ of $(x_n)$ with $x_{n_k}\to^w x$. Crucially, we rely on the following result:

\begin{lemma}[\cite{Jos1994}, see also \cite{KP2008}]\label{lem:weakComp}
Any bounded sequence in $X$ has a weak cluster point.
\end{lemma}

A function $f:X\to\mathbb{R}$ is now called weakly lower semicontinuous (weakly lsc) at a point $x\in X$ if
\[
\liminf_{n\to\infty} f(x_n)\geq f(x)
\]
whenever $(x_n)$ is a sequence such that $x_n\to^w x$. Crucially, if $f:X\to\mathbb{R}$ is a convex lsc function on a Hadamard space, it is also weakly lsc (see \cite{Bac2013}) and this result rather immediately extends to quasiconvex functions. For the proof, we in particular need the following result that closed convex subsets of Hadamard spaces are weakly sequentially closed:

\begin{lemma}[\cite{BSS2012}]\label{lem:weaklyClosed}
Let $C\subseteq X$ be a closed convex set. If $(x_n)\subseteq C$ and $x_n\to^w x$ for some $x\in X$, then $x\in C$.
\end{lemma}

We then get the following result (which is a simple extension of Lemma 3.1 in \cite{Bac2013}):

\begin{lemma}\label{lem:lscWeak}
Let $f:X\to\mathbb{R}$ be a quasiconvex lsc function. Then $f$ is weakly lsc.
\end{lemma}
\begin{proof}
Suppose for a contradiction that there is a sequence $(x_n)\subseteq X$ and $x\in X$ with $x_n\to^w x$ and
\[
\liminf_{n\to\infty} f(x_n)<f(x),
\]
i.e.\ there exists a subsequence $(x_{n_k})$, a $k_0\in\mathbb{N}$ and an $\varepsilon>0$ with $f(x_{n_k})\leq f(x)-\varepsilon$ for all $k\geq k_0$. By quasiconvexity of $f$, we get
\[
f(y)\leq f(x)-\varepsilon
\]
for all $y\in\mathrm{co}\{x_{n_k}\mid k\geq k_0\}$, where $\mathrm{co}\{x_{n_k}\mid k\geq k_0\}$ is the convex hull of $\{x_{n_k}\mid k\geq k_0\}$. By lower semicontinuity of $f$, we have the same for $y\in \overline{\mathrm{co}}\{x_{n_k}\mid k\geq k_0\}$, the closed convex hull of the respective set. But by Lemma \ref{lem:weaklyClosed}, we have $x\in \overline{\mathrm{co}}\{x_{n_k}\mid k\geq k_0\}$ which yields a contradiction.
\end{proof}

This allows us to show the following result:

\begin{lemma}\label{lem:proxNE}
Let $f:X\to\mathbb{R}$ be a quasiconvex lsc function which is bounded below. Then for any $x\in X$ and $\beta>0$, $\mathrm{Prox}_{\beta f}(x)$ is nonempty.
\end{lemma}

\begin{proof}
As commented on before, the function $g(y):=f(y)+\frac{1}{2\beta}d^2(x,y)$ is supercoercive and by assumption $f$, and hence $g$, are bounded below. Take a sequence $(y_n)$ such that $g(y_n)\to \inf g$. By Lemma \ref{lem:supercoBound}, $(y_n)$ is bounded and so by Lemma \ref{lem:weakComp}, $(y_n)$ has a weak cluster point $y$, i.e.\ there exists a subsequence $(y_{n_k})$ such that $y_{n_k}\to^w y$. As $f$ is quasiconvex and lower semicontinuous, by Lemma \ref{lem:lscWeak} we get that $f$ is weakly lower semicontinuous. Also, by \eqref{cn_ineq}, $\frac{1}{2\beta}d^2(x,y)$ is in particular convex and clearly lower semicontinuous, and so also weakly lower semicontinuous. Therefore $g$ is weakly lower semicontinuous and we have
\[
g(y)\leq\liminf_{k\to\infty}g(y_{n_k})=\inf g
\]
as $g(y_{n_k})\to \inf g$. Thus $y$ is a minimizer of $g$ and so clearly $y\in \mathrm{Prox}_{\beta f}(x)$.
\end{proof}

Now, the main property of the proximal map of a strongly quasiconvex function that we rely on is the following, which extends an analogous result due to Lara \cite{Lar2022} (see Proposition 7 therein) from $\mathbb{R}^n$ to Hadamard spaces:

\begin{lemma}\label{lem:ProxRes}
Let $f:X\to\mathbb{R}$ be a strongly quasiconvex function with value $\gamma>0$ and let $\beta>0$ and $x\in X$. If $\overline{x}\in\mathrm{Prox}_{\beta f}(x)$, then
\[
f(\overline{x})\leq\max\{f(y),f(\overline{x})\}+\frac{\lambda}{2}\left(\frac{\lambda}{\beta}-\gamma+\lambda\gamma\right)d^2(y,\overline{x})+\frac{\lambda}{\beta}\ql{x\overline{x}}{\overline{x}y}
\]
for all $y\in X$ and $\lambda\in[0,1]$.
\end{lemma}
\begin{proof}
If $\overline{x}\in\mathrm{Prox}_{\beta f}(x)$, then by definition
\[
f(\overline{x})+\frac{1}{2\beta}d^2(\overline{x},x)\leq f(z)+\frac{1}{2\beta}d^2(z,x)
\]
for all $z\in X$. Now, given $y\in X$ and $\lambda\in [0,1]$, we set $z=(1-\lambda)\overline{x}\oplus \lambda y$ in the above to obtain
\begin{align*}
f(\overline{x})+\frac{1}{2\beta}d^2(\overline{x},x)&\leq f((1-\lambda)\overline{x}\oplus \lambda y)+\frac{1}{2\beta}d^2((1-\lambda)\overline{x}\oplus \lambda y,x)\\
&\leq \max\{f(y),f(\overline{x})\}-\lambda(1-\lambda)\frac{\gamma}{2}d^2(y,\overline{x})+\frac{1}{2\beta}d^2((1-\lambda)\overline{x}\oplus \lambda y,x)
\end{align*}
using that $f$ is strongly quasiconvex with value $\gamma$. Now, using \eqref{cn_ineq}, we get
\[
d^2((1-\lambda) \overline{x}\oplus \lambda y,x)\leq (1-\lambda) d^2(\overline{x},x)+\lambda d^2(y,x)-\lambda(1-\lambda)d^2(y,\overline{x})
\]
and so
\begin{align*}
f(\overline{x})&\leq \max\{f(y),f(\overline{x})\}-\frac{\lambda}{2}(\gamma-\lambda\gamma) d^2(y,\overline{x})+\frac{\lambda}{2}\frac{\lambda}{\beta}d^2(y,\overline{x})\\
&\qquad+\frac{1}{2\beta} d^2(\overline{x},x)-\frac{\lambda}{2\beta} d^2(\overline{x},x)+\frac{\lambda}{2\beta} d^2(y,x)-\frac{\lambda}{2\beta}d^2(y,\overline{x})-\frac{1}{2\beta}d^2(\overline{x},x)\\
&\leq \max\{f(y),f(\overline{x})\}+\frac{\lambda}{2}\left(\frac{\lambda}{\beta}-\gamma+\lambda\gamma\right) d^2(y,\overline{x})+\frac{\lambda}{\beta}\cdot\frac{1}{2}\left(d^2(y,x)-d^2(\overline{x},x)-d^2(y,\overline{x})\right)\\
&\leq \max\{f(y),f(\overline{x})\}+\frac{\lambda}{2}\left(\frac{\lambda}{\beta}-\gamma+\lambda\gamma\right) d^2(y,\overline{x})+\frac{\lambda}{\beta}\ql{x\overline{x}}{\overline{x}y}
\end{align*}
which is what we wanted to show.
\end{proof}

Before moving on to the proximal point algorithm, we want to highlight that, also in this nonlinear context, fixed points of the proximal map are exactly the minimizers of the strongly quasiconvex function. The proof given here is a simple nonlinear version of Lara's proof of an analogous result given in \cite{Lar2022} (see Proposition 9 therein).

\begin{proposition}
Let $f:X\to\mathbb{R}$ be a strongly quasiconvex function with value $\gamma>0$. Then
\[
\mathrm{Fix}(\mathrm{Prox}_{\beta f})=\mathrm{argmin}f
\]
for any $\beta>0$.
\end{proposition}
\begin{proof}
Let $\overline{x}\in \mathrm{Prox}_{\beta f}(\overline{x})$. Then
\[
f(\overline{x})= f(\overline{x})+\frac{1}{2\beta}d^2(\overline{x},\overline{x})\leq f(x)+\frac{1}{2\beta}d^2(x,\overline{x})
\]
for all $x\in X$. Taking $x=(1-\lambda)\overline{x}\oplus \lambda y$ for some arbitrary $y\neq\overline{x}$ (if $X$ consists of a single point, the result is trivial) and $\lambda\in [0,1]$ in the above then yields
\begin{align*}
f(\overline{x})&\leq f((1-\lambda)\overline{x}\oplus \lambda y)+\frac{1}{2\beta}d^2((1-\lambda)\overline{x}\oplus \lambda y,\overline{x})\\
&\leq \max\{f(\overline{x}),f(y)\}-\lambda(1-\lambda)\frac{\gamma}{2}d^2(\overline{x},y)+\frac{1}{2\beta}d^2((1-\lambda)\overline{x}\oplus \lambda y,\overline{x})\\
&= \max\{f(\overline{x}),f(y)\}-\lambda(1-\lambda)\frac{\gamma}{2}d^2(\overline{x},y)+\frac{1}{2\beta}\lambda^2 d^2(\overline{x}, y)\\
&= \max\{f(\overline{x}),f(y)\}+\frac{\lambda}{2}\left(\frac{\lambda}{\beta}-\gamma+\lambda\gamma\right)d^2(\overline{x},y)
\end{align*}
where the third line follows from the fact that
\[
d((1-\lambda)\overline{x}\oplus \lambda y,\overline{x}) = \lambda d(\overline{x}, y)
\]
holds in the hyperbolic space $X$. Taking $\lambda\in (0,\beta\gamma/(1+\beta\gamma))$ yields $\lambda/\beta-\gamma+\lambda\gamma<0$ and so
\[
f(\overline{x})<\max\{f(\overline{x}),f(y)\}
\]
which implies $f(\overline{x})<f(y)$. As $y\in X\setminus \{\overline{x}\}$ was arbitrary, we get $\overline{x}\in\mathrm{argmin}f$.

Conversely, let $\overline{x}\in \mathrm{argmin}f$. Then
\[
f(\overline{x})+\frac{1}{2\beta}d^2(\overline{x},\overline{x})=f(\overline{x})\leq f(y)\leq f(y)+\frac{1}{2\beta}d^2(\overline{x},y)
\]
for any $y\in X$ and so $\overline{x}\in\mathrm{Prox}_{\beta f}(\overline{x})$.
\end{proof}

\section{The proximal point algorithm for strongly quasiconvex functions}\label{sec:ppa}

We now are concerned with the proximal point algorithm for a strongly quasiconvex lsc function $f$ and its convergence to the minimizer, if existent.\\

Concretely, by the proximal point algorithm, we shall in the following understand the following method: Given a sequence $(c_k)$ of nonnegative reals with $\inf_{k\in\mathbb{N}}c_k\geq c>0$ and an arbitrary $x^0\in X$, define an iteration $(x^k)$ by choosing
\[
x^{k+1}\in\mathrm{Prox}_{c_kf}(x^k).\tag{PPA}\label{PPA}
\]
We assume throughout that $\mathrm{argmin}f\neq\emptyset$. By Lemma \ref{lem:minUnique}, we have that $\mathrm{argmin}f$ is a singleton and we denote the unique minimizer by $x^*$ in the following. By Lemma \ref{lem:proxNE}, the sequence $(x^k)$ is well-defined.\\

The first crucial result is that the sequence $(x^k)$ defined by \eqref{PPA} satisfies the following recursive inequality (modeled, together with its proof, after Eq.\ (4.5) in \cite{Lar2022}), which in particular implies that $(x^k)$ is Fej\'er monotone w.r.t.\ $\mathrm{argmin}f$:

\begin{lemma}\label{lem:extFejMon}
Let $(x^k)$ be defined by \eqref{PPA}. Then for any $k\in\mathbb{N}$:
\[
d^2(x^{k+1},x^*)\leq d^2(x^k,x^*)-d^2(x^k,x^{k+1}).
\]
\end{lemma}
\begin{proof}
Fix $k\in\mathbb{N}$. Using that $x^{k+1}\in\mathrm{Prox}_{c_kf}(x^k)$, Lemma \ref{lem:ProxRes} implies
\[
f(x^{k+1})\leq\max\{f(y),f(x^{k+1})\}+\frac{\lambda}{2}\left(\frac{\lambda}{c_k}-\gamma+\lambda\gamma\right)d^2(y,x^{k+1})+\frac{\lambda}{c_k}\ql{x^kx^{k+1}}{x^{k+1}y}
\]
for all $y\in X$ and $\lambda\in[0,1]$. Take $y=x^*$. Then we get
\begin{align*}
0&= f(x^{k+1})-\max\{f(x^*),f(x^{k+1})\}\\
&\leq \frac{\lambda}{2}\left(\frac{\lambda}{c_k}-\gamma+\lambda\gamma\right)d^2(x^{k+1},x^*)+\frac{\lambda}{c_k}\ql{x^kx^{k+1}}{x^{k+1}x^*}
\end{align*}
for all $\lambda\in [0,1]$ and so (using the characterizing properties of the quasi-linearization function), we have
\[
\ql{x^kx^{k+1}}{x^*x^{k+1}}\leq \frac{c_k}{2}\left(\frac{\lambda}{c_k}-\gamma+\lambda\gamma\right)d^2(x^{k+1},x^*)
\]
whenever $\lambda>0$. Now, Lemma \ref{lem:quasiLinLem} together with the characterizing properties of the quasi-linearization function imply
\begin{align*}
d^2(x^{k+1},x^*)&=d^2(x^{k+1},x^k) + d^2(x^k,x^*) + 2\ql{x^{k+1}x^k}{x^kx^*}\\
&=d^2(x^{k+1},x^k) + d^2(x^k,x^*) + 2\ql{x^{k+1}x^k}{x^kx^{k+1}}+2\ql{x^{k+1}x^k}{x^{k+1}x^*}\\
&=d^2(x^{k+1},x^k) + d^2(x^k,x^*) - 2d^2(x^{k+1},x^k)+2\ql{x^{k+1}x^k}{x^{k+1}x^*}\\
&\leq d^2(x^k,x^*) - d^2(x^{k+1},x^k)+c_k\left(\frac{\lambda}{c_k}-\gamma+\lambda\gamma\right)d^2(x^{k+1},x^*)
\end{align*}
for any $\lambda\in (0,1]$. Taking $\lambda\in (0,\gamma c_k/(1+\gamma c_k))$, we get $c_k\left(\frac{\lambda}{c_k}-\gamma+\lambda\gamma\right)<0$ so that
\[
d^2(x^{k+1},x^*)\leq d^2(x^k,x^*)-d^2(x^k,x^{k+1}).\qedhere
\]
\end{proof}

The above lemma not only yields that the sequence $(x^k)$ defined by \eqref{PPA} is Fej\'er monotone but also that $\lim_{k\to\infty}d^2(x^k,x^{k+1})=0$. For the construction of our rate of convergence, we need a quantitative version of the special case that $\liminf_{k\to\infty}d^2(x^k,x^{k+1})=0$ which we construct in the following. For that, we first need the following abstract quantitative result on sequences of real numbers.

\begin{lemma}[folklore]\label{lem:descProp}
Let $(a_k)$ be a sequence of nonnegative reals and let $b\geq a_0$. Then for any $\varepsilon>0$:
\[
\exists k\leq \ceil*{\frac{b}{\varepsilon}}\left( a_k - a_{k+1}<\varepsilon\right).
\]
\end{lemma}
\begin{proof}
Suppose the claim is false, i.e.\ there is an $\varepsilon>0$ such that 
\[
a_k - a_{k+1}\geq \varepsilon
\]
for all $k\leq \ceil*{\frac{b}{\varepsilon}}$. Then
\[
a_0\geq a_0-a_{\ceil*{\frac{b}{\varepsilon}}+1}=\sum_{k=0}^{\ceil*{\frac{b}{\varepsilon}}}(a_k-a_{k+1})\geq \left(\ceil*{\frac{b}{\varepsilon}}+1\right)\varepsilon\geq b+\varepsilon>b
\]
which is a contradiction.
\end{proof}

\begin{lemma}\label{lem:approxProp}
Let $b\geq d^2(x^0,x^*)$. Then for any $\varepsilon>0$:
\[
\exists k\leq \ceil*{\frac{b}{\varepsilon}}\left( d^2(x^k,x^{k+1})<\varepsilon\right).
\]
\end{lemma}
\begin{proof}
Using Lemma \ref{lem:extFejMon}, we get 
\[
d^2(x^k,x^{k+1})\leq d^2(x^k,x^*) - d^2(x^{k+1},x^*).
\]
The result now follows from Lemma \ref{lem:descProp} using $a_k=d^2(x^k,x^*)$.
\end{proof}

The preceding results then allow us to give the following elementary proof of a rate of convergence for the iteration $(x^k)$ and hence in particular also for the first time derive the convergence of it towards the minimizer in the general setting of Hadamard spaces.

\begin{theorem}\label{thm:rateOfConv}
Let $X$ be a Hadamard space and let $f:X\to\mathbb{R}$ be a strongly quasiconvex lsc function with value $\gamma>0$ and where $\mathrm{argmin}f\neq\emptyset$. Let $(c_k)$ be a sequence of nonnegative reals with $\inf_{k\in\mathbb{N}}c_k\geq c>0$. Then the sequence $(x^k)$ defined by \eqref{PPA} converges to the (unique) minimizer $x^*$ of $f$ and $f(x^k)$ converges to $\min f$. Further:
\[
\forall \varepsilon>0\ \forall k\geq \varphi(\varepsilon)\left(d(x^k,x^*)<\varepsilon\right)
\]
as well as
\[
\forall \varepsilon>0\ \forall k\geq\varphi(\sqrt{2c\varepsilon})+1\left(f(x^k)-\min f<\varepsilon\right)
\]
where
\[
\varphi(\varepsilon):=\ceil*{\frac{4b}{\varepsilon^2((1-\lambda_0)\gamma c-\lambda_0)^2}}+1
\]
with $\lambda_0=\frac{1}{2}\frac{\gamma c}{1+\gamma c}$ and $b\geq d^2(x^0,x^*)$.
\end{theorem}
\begin{proof}
It suffices to show the quantitative results. Given $\varepsilon>0$, take
\[
k\leq \ceil*{\frac{4b}{\varepsilon^2((1-\lambda_0)\gamma c-\lambda_0)^2}}
\]
such that 
\[
d(x^{k+1},x^k)<\frac{1}{2}((1-\lambda_0)\gamma c-\lambda_0)\varepsilon
\]
using Lemma \ref{lem:approxProp}. As $x^{k+1}\in\mathrm{Prox}_{c_kf}(x^k)$ and as $c_k\geq c>0$, we get
\[
f(x^{k+1})\leq\max\{f(y),f(x^{k+1})\}+\frac{\lambda}{2}\left(\frac{\lambda}{c_k}-\gamma+\lambda\gamma\right)d^2(y,x^{k+1})+\frac{\lambda}{c_k}\ql{x^kx^{k+1}}{x^{k+1}y}
\]
for all $y\in X$ and $\lambda\in[0,1]$ from Lemma \ref{lem:ProxRes}. Taking $y=x^*$, we in particular have
\[
\frac{\lambda}{2}\left((1-\lambda)\gamma-\frac{\lambda}{c_k}\right)d^2(x^{k+1},x^*)\leq \frac{\lambda}{c_k}\ql{x^kx^{k+1}}{x^{k+1}x^*}
\]
and hence for $\lambda>0$ we get
\[
\frac{1}{2}\left((1-\lambda)\gamma c-\lambda\right)d^2(x^{k+1},x^*)\leq \ql{x^kx^{k+1}}{x^{k+1}x^*}\leq d(x^{k+1},x^k)d(x^{k+1},x^*).
\]
using \eqref{CS}. Now, either $d(x^{k+1},x^*)=0$, in which case the claim follows immediately from Lemma \ref{lem:extFejMon}, or $d(x^{k+1},x^*)>0$, in which case we have 
\[
\frac{1}{2}((1-\lambda)\gamma c-\lambda)d(x^{k+1},x^*)\leq d(x^{k+1},x^k)
\]
for any $\lambda\in (0,1]$. In particular, this holds for $\lambda=\lambda_0$ and we hence have
\[
\frac{1}{2}((1-\lambda_0)\gamma c-\lambda_0)d(x^{k+1},x^*)\leq d(x^{k+1},x^k)< \frac{1}{2}((1-\lambda_0)\gamma c-\lambda_0)\varepsilon.
\]
Since $(1-\lambda_0)\gamma c-\lambda_0>0$ holds by definition, we get
\[
d(x^{k+1},x^*)<\varepsilon,
\]
and so, using Lemma \ref{lem:extFejMon}, we in particular have $d(x^{j},x^*)<\varepsilon$ for any $j\geq k+1$.

For the second claim, note that again as $x^{k+1}\in\mathrm{Prox}_{c_kf}(x^k)$, we have
\[
f(x^{k+1})\leq f(x)+\frac{1}{2c_k}d^2(x,x^k)
\]
for any $k\in\mathbb{N}$ and any $x\in X$. In particular, for $x=x^*$, we get
\[
f(x^{k+1})-\min f \leq \frac{1}{2c}d^2(x^*,x^k).
\]
Now, using the first claim, we get that $d(x^*,x^k)<\sqrt{2c\varepsilon}$ for any $k\geq\varphi(\sqrt{2c\varepsilon})$ and so
\[
\frac{1}{2c}d^2(x^*,x^k)<\varepsilon
\]
from which the second claim follows immediately.
\end{proof}

\begin{remark}
The rate in the above Theorem \ref{thm:rateOfConv} could also have been obtained by an application of the general results on rates of convergence for Fej\'er monotone sequences under general metric regularity assumptions as developed in \cite{KLAN2019}. Concretely, by Lemma \ref{lem:extFejMon}, the iteration $(x^k)$ is Fej\'er monotone w.r.t.\ $\mathrm{argmin}f$ and the assumption that $f$ is strongly quasiconvex, which yields that the minimizer of $f$ is unique, can be used to construct a so-called modulus of uniqueness, which is a particular instance of a modulus of regularity as introduced in \cite{KLAN2019} (see in particular the comprehensive discussions in \cite{KLAN2019} on such issues). In fact, the strategy of the proof of Theorem \ref{thm:rateOfConv} is closely modeled after the proof of Theorem 4.1 in \cite{KLAN2019} but the above presentation allows us simultaneously to be more self-contained and to incorporate various slight optimizations into the rate and the argument.
\end{remark}

\begin{remark}
The above Theorem \ref{thm:rateOfConv} in particular implies Theorem 10 in \cite{Lar2022} since any (closed and convex subset of a) Euclidean space is a Hadamard space and since a strongly quasiconvex lsc function in Euclidean space has a minimizer which follows from Theorem 1 in \cite{Lar2022}. Naturally however, the above theorem is much broader as it not only applies in infinite dimensional Hilbert spaces but also in general arbitrary Hadamard spaces (recall also the discussion in the introduction). In particular, the rate of convergence from Theorem \ref{thm:rateOfConv} applies to the context of Lara's work \cite{Lar2022} which is not only, to our knowledge, the first general rate of convergence for that iteration but is even of very low complexity (being linear in the squared metric and in the convergence of the function values towards the minimizer).
\end{remark}

\begin{remark}
The proof of Theorem \ref{thm:rateOfConv} is still valid if the function is only strongly quasiconvex (i.e.\ not necessarily lsc) and if $X$ is only a $\CAT$ space (i.e.\ not necessarily complete) but in that case, we do not know if the iteration $(x^k)$ is well-defined.
\end{remark}

\noindent
{\bf Acknowledgments:}
This work benefited from conversations with Sorin-Mihai Grad, Lauren\c{t}iu Leu\c{s}tean, Nicoleta Dumitru and Ulrich Kohlenbach. The author was supported by the `Deutsche Forschungs\-gemein\-schaft' Project DFG KO 1737/6-2.

\bibliographystyle{plain}
\bibliography{ref}

\end{document}